\documentclass[11pt]{amsart} 
\usepackage[english]{babel}
\usepackage[hmargin=2cm,vmargin=3.8cm]{geometry} 
\usepackage{mathtools}
\usepackage{amssymb}
\usepackage{amsmath}
\usepackage{amsthm}
\usepackage{tikz}
\usepackage{tikz-cd}
\usepackage{cases}
\usepackage{enumerate}
\usepackage{scalerel,stackengine}
\usepackage{mathabx,epsfig}
\usepackage{mathrsfs}
\usepackage{svg}
\usepackage{setspace}
\usepackage[backend=biber, style=alphabetic, sorting=ynt
]{biblatex}
\usepackage{csquotes}
\usepackage{hyperref}
\usepackage{thmtools}
\usepackage{thm-restate}
\usepackage{cleveref}
\usepackage{epstopdf}

\theoremstyle{definition}
  \newtheorem{thm}{Theorem}[section]
  \newtheorem{theorem}[thm]{Theorem}
  \newtheorem{definition}[thm]{Definition}
  \newtheorem{example}[thm]{Example}
  \newtheorem{corollary}[thm]{Corollary}
  \newtheorem{proposition}[thm]{Proposition}
  \newtheorem{lemma}[thm]{Lemma}
  \newtheorem{remark}[thm]{Remark}
  \newtheorem{claim}[thm]{Claim}

\title{Abelian Cycles in the Homology of the Torelli Group}
\author{Erik Lindell }
\date{}

\begin{document}

\maketitle
\begin{abstract}
    In the early 1980's, Johnson defined a homomorphism $\mathcal{I}_{g}^1\to\bigwedge^3 H_1(S_{g},\mathbb{Z})$, where $\mathcal{I}_{g}^1$ is the Torelli group of a closed, connected and oriented surface of genus $g$ with a boundary component and $S_g$ is the corresponding surface without a boundary component. This is known as the \textit{Johnson homomorphism}. 
    
    We study the map induced by the Johnson homomorphism on rational homology groups and apply it to abelian cycles determined by disjoint bounding pair maps, in order to compute a large quotient of $H_n(\mathcal{I}_{g}^1,\mathbb{Q})$ in the stable range. This also implies an analogous result for the stable rational homology of the Torelli group $\mathcal{I}_{g,1}$ of a surface with a marked point instead of a boundary component. Further, we investigate how much of the image of this map is generated by images of such cycles and use this to prove that in the pointed case, they generate a proper subrepresentation of $H_n(\mathcal{I}_{g,1})$ for $n\ge 2$ and $g$ large enough.
\end{abstract}

\parskip=2pt  

\tableofcontents

\parskip=10pt

\section{Introduction}

\subsection{The Torelli group} Let $S_{g,r}^{s}$ denote a connected and oriented surface of genus $g$, with $r$ marked points and $s$ boundary components. If either of $r$ or $s$ is zero we simply omit the corresponding index. Recall that the \textit{mapping class group} of $S_{g,r}^s$ is the group
$$\Gamma_{g,r}^s:=\pi_0\mathrm{Diff}^+(S_{g,r}^s)$$
of isotopy classes of orientation-preserving diffeomorphisms of $S_{g,r}^s$ that fix the marked points and boundary components. We let $H:=H_1(S_{g},\mathbb{Q})$ and $H_{\mathbb{Z}}:=H_1(S_{g},\mathbb{Z})$. 
Since $\mathrm{Diff}^+(S_{g,r}^s)$ acts on $S_g$, $\Gamma_{g,r}^s$ acts on $H$, which is a symplectic vector space. The $\Gamma_{g,r}^s$-action preserves the corresponding symplectic form 
$$\omega:=\sum_{i=1}^g a_i\wedge b_i\in\bigwedge ^2 H,$$
where $a_1,b_1$, $\ldots$, $a_{g},b_g$ is any symplectic basis of $H$, and this gives us a group homomorphism
$$\Gamma_{g,r}^s\to\mathrm{Sp}(H),$$
where $\mathrm{Sp}(H)$ denotes the symplectic group of $H$. We define the \textit{Torelli group} of $S_{g,r}^s$, denoted $\mathcal{I}_{g,r}^s$, to be the kernel of this homomorphism. The Torelli group is thus the subgroup of $\Gamma_{g,r}^s$ consisting of those isotopy classes of diffeomorphisms that act trivially on $H$. It is a classic result that the image of $\Gamma_{g,r}^s\to\mathrm{Sp}(H)$ is precisely the arithmetic subgroup $\mathrm{Sp}(H_{\mathbb
{Z}})$, so we get a short exact sequence
$$1\to\mathcal{I}_{g,r}^s\to\Gamma_{g,r}^s\to\mathrm{Sp}(H_{\mathbb{Z}})\to 1.$$
The conjugation action of $\Gamma_{g,r}^s$ on $\mathcal{I}_{g,r}^s$ gives us an action by $\mathrm{Sp}(H_{\mathbb{Z}})$ on $\mathcal{I}_{g,r}^s$ by outer automorphisms, which makes $H_n(\mathcal{I}_{g,r}^s,\mathbb{Q})$ into a $\mathrm{Sp}(H_{\mathbb{Z}})$-representation, for all $n\ge 0$. Unless explicitly stated, all homology is taken with rational coefficients from now on.

\subsection{The Johnson homomorphism} Many basic questions about the groups $H_n(\mathcal{I}_{g,r}^s)$ are unanswered for $n\ge 2$. By contrast, the case $n=1$ is completely understood, due to work by Johnson in the early 1980's \cite{Johnson-Survey}. Among other things, Johnson contructed a group homomorphism
$$\tau_J:\mathcal{I}_{g}^1\to\bigwedge^3 H_{\mathbb{Z}}$$
which is known as the \textit{Johnson homomorphism}. For $g\ge 3$, he proved that the induced $\mathrm{Sp}(H_{\mathbb{Z}})$-equivariant map
$$\tau:H_1(\mathcal{I}_{g}^1)\to \bigwedge^3 H$$
is an isomorphism.  In particular, this proves that $H_1(\mathcal{I}_{g}^1)$ is finite dimensional and algebraic as a representation of $\mathrm{Sp}(H_{\mathbb{Z}})$ for $g\ge 3$. Two important open questions about the higher homology groups are thus:

\textbf{Question 1.} Are the homology groups $H_n(\mathcal{I}_{g,r}^s)$ finite dimensional, for $n\ge 1$ and $g$ sufficiently large?

\textbf{Question 2.} Are the homology groups $H_n(\mathcal{I}_{g,r}^s)$ algebraic $\mathrm{Sp}(H_{\mathbb{Z}})$-representations, for $n\ge 1$ and $g$ sufficiently large?

In low genera, we know that the answer to both questions is generally \textit{no}. In genus 2, it was proven by Mess \cite{Mess-Torelli} that for a certain permutation matrix $e\in\mathrm{Sp}(2,\mathbb{Z})$ we have
$$H_1(\mathcal{I}_2)\cong\mathbb{Q}[\mathrm{Sp}(2,\mathbb{Z})/(\langle e\rangle\ltimes(\mathrm{SL}(2,\mathbb{Z})\times\mathrm{SL}(2,\mathbb{Z})))].$$  He also showed that $H_3(\mathcal{I}_3)$ contains a subrepresentation isomorphic to $\mathbb{Q}[\mathrm{Sp}(3,\mathbb{Z})/(\mathrm{Sp}(1,\mathbb{Z})\times\mathrm{Sp}(2,\mathbb{Z}))]$. It was recently proven by Gaifullin \cite{GaifullinTorelli} that for $g\ge 3$ and $2g-3\le n\le 3g-6$, $H_n(\mathcal{I}_g,\mathbb{Z})$ contains a free abelian subgroup of infinite rank. This means that $H_n(\mathcal{I}_g)$ is neither algebraic nor finite dimensional in either of these cases.

For $n\ge 1$, $\tau_J$ induces a map
$$(\tau_J)_*:H_n(\mathcal{I}_{g}^1)\to H_n\left(\bigwedge^3 H\right)\cong\bigwedge^n\left(\bigwedge^3 H\right),$$
which we will denote by $\psi_n$. An equivalent definition of $\psi_n$ is as the composition of the comultiplication $H_n(\mathcal{I}_{g}^1)\to\bigwedge^n H_1(\mathcal{I}_{g}^1)$ with $\wedge^n \tau$. For $n\ge 2$, it follows from results by Hain that the map $\psi_n$ is not surjective \cite{Hain-Infinitesimal}, but a natural question to ask is:

\textbf{Question 3.} Is $\psi_n$ injective for all $n\ge 1$ and $g$ sufficiently large?

It is not clear whether to expect this to be true. A positive answer would of course imply positive answers to Questions 1 and 2, in the corresponding case. It was proven by Kupers and Randal-Williams that the converse is also true \cite{Kupers-Torelli}. They also gave an explicit description of the kernel of the dual map of $\psi_n$, under these assumptions.

\begin{remark} In the case of a surface without a boundary component, Johnson instead defined a homomorphism $\mathcal{I}_g\to\left(\bigwedge^3 H_{\mathbb{Z}}\right)/H_{\mathbb{Z}},$ where $H_{\mathbb{Z}}$ is embedded as a subspace of $\bigwedge^3H_{\mathbb{Z}}$ by inserting the symplectic form. This is also known as the Johnson homomorphism. This homomorphism induces an $\mathrm{Sp}(H_{\mathbb{Z}})$-equivariant isomorphism $H_1(\mathcal{I}_g)\to \left(\bigwedge^3 H\right)/H$. \end{remark}

\begin{remark}Hain has studied the dual of the map $\psi_n$ in the unpointed case and determined its image completely for $n=2$ \autocite{Hain-Infinitesimal}. For $n=3$, its image was determined by Sakasai, up to one irreducible subrepresentation \autocite{Sakasai-Torelli}, the presence of which was settled by Kupers and Randal-Williams \cite{Kupers-Torelli}.

The Johnson homomorphisms are closely related to homomorphisms $\Gamma_{g,1}\to\frac{1}{2}\bigwedge^3 H_{\mathbb{Z}}\rtimes\mathrm{Sp}(H_{\mathbb{Z}})$ and $\Gamma_g\to\frac{1}{2}\left(\bigwedge^3 H_{\mathbb{Z}}\right)/H_{\mathbb{Z}}\rtimes \mathrm{Sp}(H_{\mathbb{Z}})$ defined by Morita \cite{MoritaExtension}, which were used by Kawazumi and Morita to give a description of the tautological subalgebra of $H^*(\Gamma_{g,1})$ in terms of trivalent graphs \cite{KawazumiPrimary}.
\end{remark}

We will study how the image of $\psi_n$ decomposes into irreducible $\mathrm{Sp}(H_{\mathbb{Z}})$-representations and determine a lower bound on this image in the stable range. The representation $\bigwedge^n(\bigwedge^3 H)$ is algebraic and this implies that its irreducible subrepresentations viewed as an $\mathrm{Sp}(H)$- and an $\mathrm{Sp}(H_{\mathbb{Z}})$-representation agree. Therefore, we will work with the image of $\psi_n$ as an $\mathrm{Sp}(H)$-representation. 

\subsection{Irreducible representations of symplectic groups}\label{secIrreps} In order to state our results, we need to recall some basics from the representation theory of symplectic groups. The isomorphisms classes of irreducible representations of $\mathrm{Sp}(H)$ are indexed by \textit{partitions}. A partition is a decreasing sequence $(\lambda_1\ge\lambda_2\ge\cdots\ge 0\ge 0\ge \cdots)$ of non-negative integers that eventually reaches zero. For example, the standard representation $H$ corresponds to the partition $(1\ge 0\ge 0\ge\cdots)$. We will write $(\lambda_1\ge\lambda_2\ge\cdots\ge\lambda_k)$ for the partition where all following entries in the sequence are zero. The sum $\sum_{n\ge 1}\lambda_n$ is called the \textit{weight} of the partition.  If $\lambda_{i-1}>\lambda_{i}=\cdots=\lambda_{i+l-1}>\lambda_{i+l}$, we will often write $(\lambda_1\ge\ldots>\lambda_i^l>\ldots\ge\lambda_k)$, for brevity, and use the convention that if $l=0$, then $\lambda_i$ does not occur in the partition.

We will write $V_\lambda=V_{\lambda_1,\ldots,\lambda_k}$ for the irreducible representation corresponding to the partition $\lambda=(\lambda_1\ge\lambda_2\ge\cdots\ge\lambda_k)$. In particular, we have $V_1\cong H$. We define the weight of $V_{\lambda}$ to be the weight of $\lambda$. For $k\ge 1$, the representation $H^{\otimes k}$ contains irreducible subrepresentations of weight at most $k$, which means that the top weight irreducible subrepresentations of $\bigwedge^n(\bigwedge^3 H)$ have weight $3n$. For more details on the irreducible representations of symplectic groups, see for example \cite[Chapter 17]{Fulton-Harris}.

\subsection{Results} Our first main result is the following:

\begin{theorem}\label{theorem1}
For $n\ge 1$ and $g\ge 3n$, the image of $\psi_n:H_n(\mathcal{I}_{g}^1)\to\bigwedge^n\left(\bigwedge^3 H\right)$ contains all irreducible $\mathrm{Sp}(H)$-subrepresentations of $\bigwedge^n\left(\bigwedge^3 H\right)$ of weight $3n$.
\end{theorem}

\begin{remark}
The kernel of the dual map described by Kupers and Randal-Williams contains no irreducible subrepresentations of top weight, which means that if Questions 1 and 2 have positive answers, Theorem \ref{theorem1} follows immediately. 
\end{remark}

To get a feeling for what Theorem \ref{theorem1} means in practice, let us list the irreducible subrepresentations of $\bigwedge^n\left(\bigwedge^3 H\right)$ of top weight for $n\le 4$ in the table below, together with their dimension for $g=3n$. These can be computed using, for example, SAGE.

\begin{center}
 \begin{tabular}{|c|l|c|} 
 \hline
 $n$ & Irreducible subrepresentations of top weight in $\bigwedge^n\left(\bigwedge^3 H\right)$, for $g\ge 3n$ & Dimension for $g=3n$ \\
 \hline
 1 & $V_{1^3}$&  14 \\ [0.5ex]

 2 & $V_{2^2,1^2}\oplus V_{1^6}$ &19383\\[0.5ex]
 
 3 & $V_{3,2^3}\oplus V_{3^2,1^3}\oplus V_{2^3,1^3}\oplus V_{2^2,1^5}\oplus V_{1^9}$ &$\approx 7.5\cdot 10^7$ \\[0.5ex]
 
 4 &  $V_{4,3,2^2,1}\oplus V_{4^2,1^4}\oplus  V_{3^4}\oplus V_{3^2,2^2,1^2}\oplus V_{3^2,1^6}\oplus V_{3,2^3,1^3}$& $\approx5.3\cdot 10^{11}$\\[0.2ex]
 & $\oplus V_{2^6}\oplus V_{2^4,1^4}\oplus V_{2^3,1^3}\oplus V_{2^2,1^8}\oplus V_{1^{12}}$& \\[1ex]
 \hline
\end{tabular}
\end{center}

Contrary to what this table might suggest, the irreducible subrepresentations are generally not of multiplicity one for higher $n$.

Theorem \ref{theorem1} is actually a special case of our second main theorem. This second theorem requires a bit more work to prove, which is why we list and prove the theorems separately. In order to state this result a bit more clearly, we will consider the irreducible $\mathrm{Sp}(H)$-representation $V_{1^k}$ as a $\mathbb{Z}$-graded vector space concentrated in degree $k$, and use the Koszul sign convention. Note that in particular, $H$ is concentrated in degree 1. From now on, we will therefore use $\bigwedge$ to denote the free graded commutative algebra functor, rather than the exterior algebra functor of ungraded vector spaces.

\begin{restatable}{thm}{theoremtwo}\label{theorem2}
Let $n\ge 1$ and $\lambda=(\lambda_1^{k_1}>\cdots> \lambda_m^{k_m}>0)$ and $\mu=(\mu_1^{l_1}>\cdots>\mu_{m+2}^{l_{m+2}}>0)$ be partitions such that $\mu_{i}=\lambda_i+2$ for
$i=1,\ldots,m$ and $|\lambda|+|\mu|=n$. Let $k_{m+1}=k_{m+2}=0$, $k=\sum_{i=1}^{m} k_i$, $l=\sum_{i=1}^{m+2}l_i$ and suppose that $g\ge n+2k+l$. Then every irreducible subrepresentation of
$$\bigotimes_{i=1}^{m+2}\bigwedge^{k_i+l_i}V_{1^{\mu_i}}$$
of weight $n+2k$ lies is $\mathrm{Im}(\psi_n)$.
\end{restatable}

In particular, we see that the case $m=1$, $\lambda_1=1$, $k_1=n$ gives us back Theorem \ref{theorem1}, since $\bigwedge^3 H\cong V_{1^3}\oplus V_1$.

\begin{remark}\label{remarkpointorboundary}
The Johnson homomorphism $\tau_J:\mathcal{I}_{g}^1\to\bigwedge^3 H_\mathbb{Z}$ actually factors as $\mathcal{I}_{g}^1\to\mathcal{I}_{g,1}\to\bigwedge^3 H_\mathbb{Z}$, where the first map is induced by sewing a disk with a marked point into the boundary component and extending diffeomorphisms by the identity. Thus the analogous results to Theorems \ref{theorem1} and \ref{theorem2} also hold in the pointed case. \end{remark}
 
Our strategy is to consider \textit{abelian cycles} in $H_n(\mathcal{I}_{g}^1)$ determined by $n$-tuples of pairwise disjoint \textit{bounding pair maps}. These notions will be defined in Section 3. What makes this strategy work is that the coproduct of an abelian cycle has a simple expression in terms of abelian cycles of lower degree (Lemma \ref{lemmaComult}), which makes the map $\psi_n$ easy to evaluate.

\begin{remark} The main inspiration for using abelian cycles are papers by Church and Farb \autocite{Church-Farb-Torelli} and Sakasai \autocite{Sakasai-Torelli}. In the former, these are used to study the images of certain homomorphisms $\tau_n:H_n(\mathcal{I}_{g,1})\to \bigwedge^{n+2} H$ that generalize $\tau$ to higher degree (see \cite[pp. 172-173]{Johnson-Survey}). Church and Farb prove, among other things, that for $n\ge 1$ and $g\ge n+2$, the homomorphism $\tau_n$ gives a surjection $H_n(\mathcal{I}_{g,1})\twoheadrightarrow V_{1^{n+2}}\oplus V_{1^n}$. This means that the dimension of $H_n(\mathcal{I}_{g,1})$ has a lower bound of order $\sim g^{n+2}$ , for all $n\ge 1$. We can note that by Theorem \ref{theorem1} and  Remark \ref{remarkpointorboundary} we get a lower bound on the dimensions of $H_n(\mathcal{I}_{g}^1)$ and $H_n(\mathcal{I}_{g,1})$ of order $\sim g^{3n}$ for $g\ge 3n$. 
\end{remark}

\begin{remark} Church, Ellenberg and Farb have used representation stability \cite[Section 7.2]{ChurchEllenbergFarb} to prove the following theorem:
\begin{theorem}[Church, Ellenberg, Farb] \label{thmCEF}
For each $n\ge 0$ there exists a polynomial $P_n(T)$ of degree at most $3n$ such that
$$\mathrm{dim}\left(\mathrm{Im}\left(\bigwedge^n H^1(\mathcal{I}_{g}^1)\to H^n(\mathcal{I}_{g}^1)\right)\right)=P_n(g).$$
for $g\gg n$.
\end{theorem}
Since $H^n(\mathcal{I}_{g}^1)\cong H_n(\mathcal{I}_{g}^1)^*$ as $\mathrm{Sp}(H)$-representations for all $n\ge 0$, it follows that the dimension of the image in Theorem \ref{thmCEF} is bounded from below by the dimension of the image of $\psi_n:H_n(\mathcal{I}_g^1)\to\bigwedge^n\left(\bigwedge^3 H\right)$, so the following corollary follows from Theorems \ref{theorem1} and \ref{thmCEF}:

\begin{corollary}
The polynomial $P_n(T)$ in Theorem \ref{thmCEF} has degree $3n$.
\end{corollary}
\end{remark}

For the final result of the paper, we will work with a marked point instead of a boundary component, since this gives us something a bit more general. Abusing notation, we will use $\psi_n$ to denote the map $H_n(\mathcal{I}_{g,1})\to\bigwedge^n\left(\bigwedge^3 H\right)$ as well. Which $\psi_n$ is intended will be clear from context. The last main result concerns the limitations of using abelian cycles of the type we consider. More specifically:

\begin{restatable}{thm}{theoremthree}\label{theorem3}
Let $n\ge 2$ and $A_n(\mathcal{I}_{g,1})\subset H_n(\mathcal{I}_{g,1})$ be the subrepresentation generated by abelian cycles determined by pairwise disjoint bounding pair maps. Then $\psi_n(A_n(\mathcal{I}_{g,1}))$ is concentrated in weights $n, n+2,\ldots,3n$. However, $\psi_n(H_n(\mathcal{I}_{g,1}))$ also contains a subrepresentation of weight $n-2$ for all $n\ge 2$ and $g\gg 0$. In particular, $A_n(\mathcal{I}_{g,1})\subset H_n(\mathcal{I}_{g,1})$ is a proper subrepresentation for $n\ge 2$ and $g\gg0$.
\end{restatable}

\begin{remark}
We will see that the map induced on homology by $\mathcal{I}_g^1\to\mathcal{I}_{g,1}$ maps $A_n(\mathcal{I}_g^1)\to A_n(\mathcal{I}_{g,1})$. From this it follows that the first claim of Theorem \ref{theorem3} holds for $\psi_n(A_n(\mathcal{I}_g^1))$ as well. To prove the second claim, however, we construct a homology class that requires a marked point rather than a boundary component, and it is not clear how to construct something analogous using a boundary component. However, it follows from results by Kupers and Randal-Williams \cite{Kupers-Torelli} that $H_3(\mathcal{I}_g^1)$ contains an irreducible subrepresentation of weight 1, so it seems reasonable to expect Theorem \ref{theorem3} to hold in its entirety for $A_n(\mathcal{I}_g^1)$ as well. If we also believe that Questions 1-3 above have positive answers, it follows from the results by Kupers and Randal-Williams that for large $n$ there are many classes in $H_n(\mathcal{I}_{g}^1)$ that are of lower weight than $n$ and that thus cannot lie in $A_n(\mathcal{I}_{g}^1)$.
\end{remark}

\subsection{Structure of the paper} In Section \ref{SchurWeyl} we give a brief description of \textit{Schur-Weyl duality for symplectic groups}, which is a fundamental result in the representation theory of symplectic groups. We then describe some simple but useful consequences of this theorem.

In Section \ref{ACandBPM} we recall some basics from the theory of mapping class groups in order to define bounding pair maps. We then introduce abelian cycles and show how $\psi_n$ can be easily evaluated on these. This allows us to give a simple proof of Theorem \ref{theorem1}.

In Section \ref{firstproofsection} we prove Theorem \ref{theorem2}. Here we see how considering surfaces with a boudnary component allows us to define a ``multiplication'' of abelian cycles and thereby prove the theorem in a systematic way. 

In the final section, Section \ref{secondproofsection}, we prove Theorem \ref{theorem3}.

\subsection*{Acknowledgement} I would like to thank my PhD supervisor, Dan Petersen, whose suggestions lead me to investigate the topics of this paper, and whose supervision and support has been invaluable throughout the project.

\section{Schur-Weyl duality for symplectic groups}\label{SchurWeyl}

We want to investigate how the image of $\psi_n:H_n(\mathcal{I}_{g}^1)\to\bigwedge^n\left(\bigwedge^3 H\right)$ decomposes into irreducible representations of $\mathrm{Sp}(H)$ and to this end we first need to understand how $H^{\otimes k}$ decomposes into irreducible subrepresentations, for any $k\ge 1$. This decomposition is well understood, by so called \textit{Schur-Weyl duality} for symplectic groups. The representation $H^{\otimes k}$ contains irreducible subrepresentation of weight at most $k$ (see Section \ref{secIrreps}) and the irreducibles of top weight form a subrepresentation which we denote by $H^{\langle k\rangle}$. We call the tensors in $H^{\langle k\rangle}$ \textit{traceless}, since these are precisely those tensors in $H^{\otimes k}$ that are mapped to zero by all of the contraction maps $C^k_{i,j}:H^{\otimes k}\to H^{\otimes k-2}$, given by contracting with the $i$th and $j$th factors with the symplectic form. The space $H^{\langle k\rangle}$ has both a left action by $\mathrm{Sp}(H)$ and a commuting right action by the symmetric group $\Sigma_k$, by permuting the tensor factors. Schur-Weyl duality\footnote{There is also a more general version that describes the decomposition of all of $H^{\otimes k}$, but this requires more background to describe and is not required for our purposes.} describes how $H^{\langle k\rangle}$ decomposes as a representation of $\mathrm{Sp}(H)\times \Sigma_k$. The irreducible representations of $\Sigma_k$ are classified by partitions of weight $k$, as described in \autocite[Chapter 4]{Fulton-Harris} and for such a partition $\lambda$ we denote the corresponding irreducible representation of $\Sigma_k$ by $\sigma_\lambda$. Schur-Weyl duality for symplectic groups, which is described for example in  \cite[Chapter 17.3]{Fulton-Harris}, then states that
$$H^{\langle k\rangle}\cong\bigoplus_{|\lambda|=k}V_\lambda\otimes\sigma_\lambda.$$
Up to isomorphism, each $V_\lambda$ can be constructed as the image of a certain \textit{Young symmetrizer} $\mathbf{c}_\lambda:H^{\langle k\rangle}\to H^{\langle k\rangle}$, which is an element of $\mathbb{Q}[\Sigma_k]$. For a proof of this fact, once again see \cite[Chapters 6 and 17.3]{Fulton-Harris}. Young symmetrizers are defined using \textit{Young diagrams}, which are diagrams that represent partitions, as illustrated in Figure \ref{figYoung}. A \textit{Young tableau} is a Young diagram with a choice of numbering of the boxes.
\begin{figure}[htbp]
  \centering
  \includegraphics[scale=0.4]{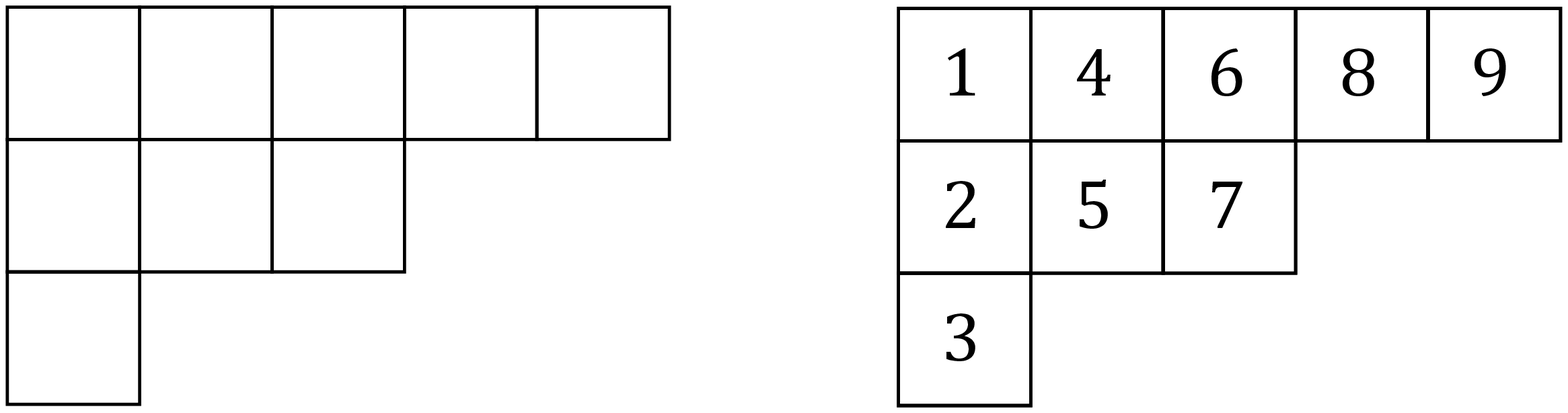}
  \caption{The Young diagram representing the partition $(5\ge 3\ge 1)$ and a corresponding Young tableau.}
  \label{figYoung}
\end{figure}
To define $\mathbf{c}_\lambda$, for $|\lambda|=k$, we first need to pick a Young tableau with an underlying Young diagram representing $\lambda$. We denote this tableau by $T$ and let $P_\lambda=\{\sigma\in\Sigma_k\mid \sigma\text{ preserves each row of }T\}$ and $Q_\lambda=\{\sigma\in\Sigma_k\mid \sigma\text{ preserves each column of }T\}$. We now define $\mathbf{c}_\lambda:=\mathbf{a}_\lambda \mathbf{b}_\lambda\in\mathbb{Q}[\Sigma_k],$
where 
$$\mathbf{a}_\lambda:=\sum_{\sigma\in P_\lambda}\sigma,\ \ \ \mathbf{b}_\lambda:=\sum_{\sigma\in Q_\lambda}\mathrm{sgn}(\sigma)\sigma.$$
If we take the Young tableau of $\lambda$ that is numbered analogously to the tableau in Figure \ref{figYoung} and let $\mu=(\mu_1\ge\cdots\ge\mu_l)$ denote the partition we get by transposing the rows and columns of the Young diagram representing $\lambda$, it follows that the image of the corresponding Young symmetrizer $\mathbf{c}_\lambda:H^{\langle k\rangle}\to H^{\langle k\rangle}$ is the subrepresentation of
$$\bigwedge^{\mu_1}H\otimes\cdots\otimes\bigwedge^{\mu_l} H$$
generated by the tensor
$$(a_1\wedge\cdots\wedge a_{\mu_1})\otimes\cdots\otimes(a_1\wedge\cdots\wedge a_{\mu_l}),$$
for any choice of symplectic basis $\{a_1,b_1,\ldots,a_g,b_g\}$ of $H$. From this, it follows that as an $\mathrm{Sp}(H)$-representation, $V_\lambda\otimes\sigma_{\lambda}$ is generated by tensors of the form
for $\sigma\in\Sigma_k$. Here the action is defined by viewing $\bigwedge^{\mu_1}H\otimes\cdots\otimes\bigwedge^{\mu_l} H$ as a subspace of $H^{\otimes k}$, i.e as the composite map
$$\bigwedge^{\mu_1}H\otimes\cdots\otimes\bigwedge^{\mu_l} H\hookrightarrow H^{\otimes k}\overset{\cdot\sigma}{\longrightarrow} H^{\otimes k},$$
where the first map is the standard inclusion. This allows us to prove the following lemma:

\begin{lemma}
For all $g\ge k\ge 1$, the $\mathrm{Sp}(H)$-representation $H^{\langle k\rangle}$ is generated by the tensor
\begin{equation}
    a_{1}\otimes\cdots \otimes a_{k},
\end{equation}
for any fixed symplectic basis $\{a_1,b_1,\ldots,a_g,b_g\}$ of $H$.
\end{lemma} 

\begin{proof}
By the discussion above it suffices to prove that every tensor $A$ of the form
$$a_{i_1}\otimes\cdots\otimes a_{i_k},$$
for $i_1,\ldots,i_j\in\{1,\ldots,g\}$, lies in the $\mathbb{Q}[\mathrm{Sp}(H)]$-span of the tensor $a_1\otimes\cdots\otimes a_{k}$. We may reorder the symplectic basis in such a way that for each $1\le j\le k$, we have $a_{i_j}=a_j$, unless $a_{i_j}=a_{i_l}$ for some $l<j$. In other words, we may assume that $A$ agrees with $a_1\otimes\cdots\otimes a_k$ in every tensor factor except for in those where $A$ has a repeating element.

For $p,q\in\{1,\ldots,g\}$ such that $p\neq q$, let $S_{p,q}\in \mathrm{Sp}(H)$ denote the linear map defined by mapping $a_p\mapsto a_p+a_q$, $b_q\mapsto b_q-b_p$ and mapping all other basis vectors to themselves. If $a_{i_j}=a_l$ for some $l<j$, we can act on $a_1\otimes\cdots \otimes a_k$ with $S_{j,l}-\mathbf{1}$ and get
$$a_1\otimes\cdots\otimes a_{j-1}\otimes a_l\otimes a_{j+1}\otimes\cdots\otimes a_k.$$
We can do this for every repeated tensor factor of $A$ and thus see that $A$ does indeed lie in the $\mathbb{Q}[\mathrm{Sp}(H)]$-span of $a_1\otimes\cdots\otimes a_k$.
\end{proof}

This has the following immediate corollary:
\begin{corollary}\label{Corollary:Main}
Let $V$ be a direct summand of $H^{\otimes k}$ for $g\ge k\ge 1$, as a representation of $\mathrm{Sp}(H)$. Let $p:H^{\otimes k}\to V$ be the projection map and fix a symplectic basis $\{a_1,b_1,\ldots,a_g,b_g\}$ of $H$. Then the subspace of $V$ spanned by irreducible subrepresentations of weight $k$ lies in the $\mathbb{Q}[\mathrm{Sp}(H)]$-span of the tensor
$$p(a_1\otimes\cdots\otimes a_{k}).$$.
\end{corollary}

\section{Abelian cycles and bounding pair maps} \label{ACandBPM}

Computing the $n$-fold coproduct of a general class $\alpha\in H_n(\mathcal{I}_{g}^1)$ is difficult, but we will restrict our interest to a simple class of homology classes, called \textit{abelian cycles}, for which we can give an explicit expression. For a group $G$, abelian cycles in $H_n(G)$ are constructed from commuting elements in the group itself, so in order to construct such classes for $\mathcal{I}_{g}^1$ in particular, we will first recall some basic notions from the theory of mapping class groups.

\subsection{Dehn twists and bounding pair maps}

Firstly, recall that if $\alpha$ is a simple, closed curve in $S_{g,r}^s$, then a \textit{Dehn twist} along $\alpha$ is a diffeomorphism of $S_{g,r}^s$ that is obtained as follows: start with some tubular neighborhood $N$ around $\alpha$, with a diffeomorphism $N\cong S^1\times I$. The corresponding Dehn twist is then given by the identity outside of $N$, and on $N\cong S^1\times I$ by $(s,t)\mapsto (se^{i2\pi t},t)$. Pictorially, we can illustrate this as in Figure \ref{figDehn}. 
\begin{figure}[htbp]
  \centering
  \includegraphics[scale=0.7]{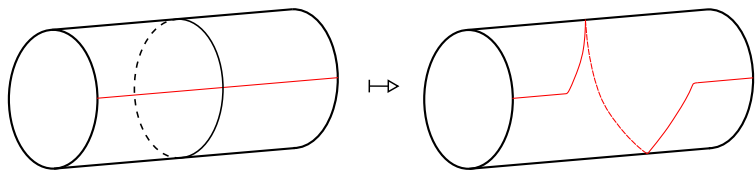}
  \caption{A Dehn twist.}
  \label{figDehn}
\end{figure}

Up to isotopy, the choice of tubular neighborhood does not matter, and if $\alpha$ and $\alpha'$ are homotopic, the corresponding Dehn twists are isotopic. Given a homotopy class $\alpha$ of simple, closed curves in $S_{g,r}^s$, it is thus well-defined to talk about \textit{the} Dehn twist along $\alpha$ as an element of $\Gamma_{g,r}^s$. We will denote this element by $T_{\alpha}$. It is a classic result by Dehn, and later independently by Lickorish, that $\Gamma_{g,r}^s$ is generated by Dehn twists. 

Dehn twists themselves are generally not elements of $\mathcal{I}_{g,r}^s$, but it is easy to find compositions of Dehn twists which are. A \textit{bounding pair} in $S_{g,r}^s$ is a pair of simple, closed, non-separating and homologous curves. If $(\alpha,\beta)$ is a bounding pair, we call the composition $T_\alpha T_{\beta^{-1}}$ the \textit{bounding pair map} corresponding to this bounding pair. The typical form of a bounding pair is illustrated in Figure \ref{figBP}. 

\begin{figure}[htbp]
  \centering
  \includegraphics[scale=0.6]{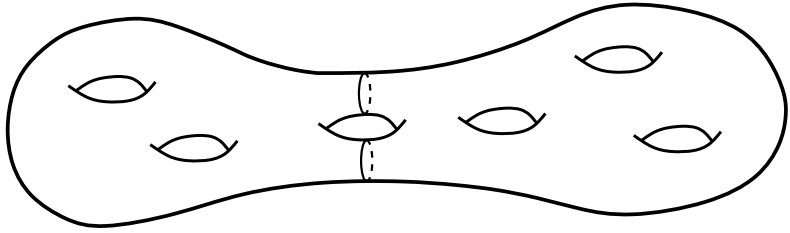}
  \caption{A typical bounding pair.}
  \label{figBP}
\end{figure}

Bounding pair maps are elements of $\mathcal{I}_{g,r}^s$ and in fact, Johnson proved that $\mathcal{I}_{g}$ is generated by a finite number of bounding pair maps, for $g\ge 3$ \cite{JohnsonFiniteGeneration}. This was also used to prove that $\tau:H_1(\mathcal{I}_{g}^1)\to \bigwedge^3 H$ is an isomorphism in this range.

Any bounding pair $(\alpha,\beta)$ separates $S_{g}^1$ into two connected components, a subsurface of genus $g_1$ with three boundary components and a genus $g_2$ subsurface with two boundary components, such that $g_1+g_2=g-1$. Let $\omega_2=\sum_{i=1}^{g_2}a_i\wedge b_i$ be the symplectic form of a maximal subsurface with only one boundary component of the connected component not containing the original boundary component. 
\begin{proposition}[Johnson]\label{tauformula}
Let $(\alpha,\beta)$ and $\omega_2$ be as above and $f$ denote the corresponding bounding pair. Then
$$\tau(f)=\omega_2\wedge[\alpha]=\sum_{i=1}^{g_2}(a_i\wedge b_i)\wedge[\alpha]\in\bigwedge^3 H,$$
where $[\alpha]$ denotes the homology class of $\alpha$.
\end{proposition}
Note that since $\alpha$ and $\beta$ belong to the same homology class, this formula is well-defined.

\subsection{Abelian cycles} If two bounding pairs are disjoint, the corresponding bounding pair maps commute. This gives us a good way to construct elements in $H_n(\mathcal{I}_{g}^1)$. If $G$ is a group and $(f_1,\ldots,f_n)$ is an $n$-tuple of pairwise commuting elements in $G$, this determines a map $\mathbb{Z}^n\to G$, which induces a map $H_n(\mathbb{Z}^n)\to H_n(G)$. We have $H_n(\mathbb{Z}^n)\cong H_n(T^n)$, where $T^n$ denotes the $n$-torus, and hence $H_n(\mathbb{Z}^n)$ is generated by its fundamental class. We call the image of this class under the map to $H_n(G)$ the \textit{abelian cycle} determined by $(f_1,\ldots,f_n)$ and denote it by $A(f_1,\ldots,f_n)$.

\begin{lemma}\label{lemmaComult}
If $\Delta^{n-1}:H_n(G)\to\bigwedge^n H_1(G)$ is the comultiplication map, then 
$$\Delta^{n-1} A(f_1,\ldots,f_n)=[f_1]\wedge\cdots\wedge [f_n],$$
where $[f]\in H_1(G)$ denotes the homology class of $f\in G$. 
\end{lemma}

\begin{proof}
The homology $H_*(\mathbb{Z}^n)$ is a bialgebra and if $\{a_1,\ldots,a_n\}$ is a basis for $H_1(\mathbb{Z}^n)$, the fundamental class of $H_n(\mathbb{Z}^n)$ is $a_1\cdots a_n$. Thus 
$$\Delta^{n-1}(a_1\cdots a_n)=a_1\wedge\cdots\wedge a_n.$$
Furthermore, the map $H_1(\mathbb{Z}^n)\to H_1(G)$ maps $a_i\mapsto [f_i]$, for $1\le i\le n$. By functoriality of homology, we have a commutative diagram
\[\begin{tikzcd}
H_n(\mathbb{Z}^n)\arrow[r]\arrow[d]& H_n(G)\arrow[d]\\
\bigwedge^n H_1(\mathbb{Z}^n)\arrow[r]& \bigwedge^n H_1(G)
\end{tikzcd}\]
which means that 
\begin{align*}
    \Delta^{n-1} A(f_1,\ldots,f_n)=[f_1]\wedge\cdots\wedge[f_n].& \qedhere
\end{align*}
\end{proof}

\begin{figure}[htbp]
  \centering
  \includegraphics[scale=0.4]{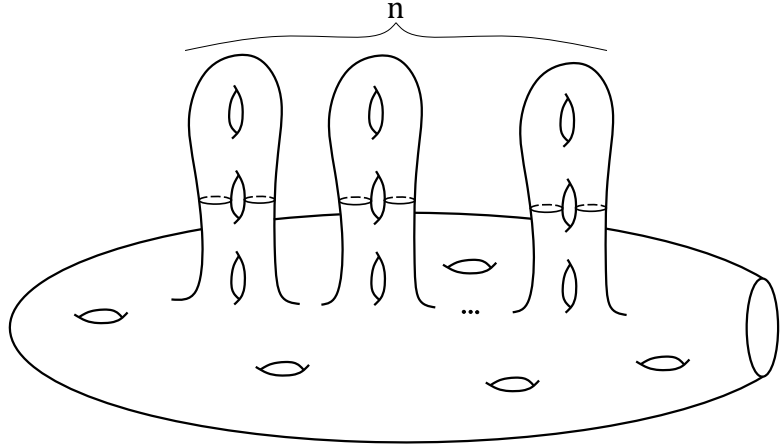}
  \caption{Bounding pairs determining an abelian cycle in $H_n(\mathcal{I}_{g}^1)$, for $g\ge 3n$.}
  \label{figAC}
\end{figure}

If $(f_1,\ldots,f_n)$ is an $n$-tuple of bounding pairs corresponding to pairwise disjoint bounding pairs, we will simply call this an $n$-tuple of disjoint bounding pair maps. We will consider abelian cycles in $H_n(\mathcal{I}_{g,r}^s)$ that are determined by such $n$-tuples and denote the subrepresentation of $H_n(\mathcal{I}_{g,r}^s)$ generated by these by $A_n(\mathcal{I}_{g,r}^s)$. 

\begin{remark}
It is clear from the definition of bounding pair map that the map $\mathcal{I}_g^1\to\mathcal{I}_{g,1}$ induces a map $A_n(\mathcal{I}_g^1)\to A_n(\mathcal{I}_{g,1})$.
\end{remark}

If $A(f_1,\ldots,f_n)\in H_n(\mathcal{I}_g^1)$ is an abelian cycle, it follows by Lemma \ref{lemmaComult} that
$$\psi_n(A(f_1,\ldots,f_n))=\tau([f_1])\wedge\cdots\wedge\tau([f_n]).$$
To evaluate $\psi_n$ on an abelian cycle, this means that we only need Johnson's original result on the image of $\tau$. 

\begin{example}\label{exampleBPs} Consider the abelian cycle $A$ determined by the bounding pairs in Figure \ref{figAC}. With an appropriate choice of symplectic basis of $H$, we see that the image of this abelian cycle is
$$\psi_n(A)=(a_1\wedge b_1\wedge a_3)\wedge\cdots\wedge (a_{3n-2}\wedge b_{3n-2}\wedge a_{3n}).$$
\end{example}

We can now use this to prove Theorem \ref{theorem1}:

\begin{proof}[Proof of Theorem \ref{theorem1}]
Since $\bigwedge^n\left(\bigwedge^3 H\right)$ is a direct summand of $H^{\otimes 3n}$, it follows by Corollary \ref{Corollary:Main} that it suffices to show that the tensor
$$(a_1\wedge a_2\wedge a_3)\wedge \cdots\wedge (a_{3n-2}\wedge a_{3n-1}\wedge a_{3n})$$
lies in the image of $\psi_n$. To this end we consider the abelian cycle from Example \ref{exampleBPs}. For $i\neq j\in\{1,\ldots,g\}$ let $T_{i,j}:H\to H$ denote the linear map defined by mapping $b_i\mapsto b_i+a_j$, $b_j\mapsto b_j+a_i$ and mapping all other basis vectors to themselves. Each such map lies in $\mathrm{Sp}(H)$ and we see that if we act with $\prod_{i=1}^{n}(T_{3i-2,3i-1}-\mathbf{1})\in\mathbb{Q}[\mathrm{Sp}(H)]$ on
$$(a_1\wedge b_1\wedge a_3)\wedge\cdots\wedge (a_{3n-2}\wedge b_{3n-2}\wedge a_{3n}),$$
we get
\begin{align*}
    (a_1\wedge a_2\wedge a_3)\wedge\cdots\wedge (a_{3n-2}\wedge a_{3n-1}\wedge a_{3n}). & \qedhere
\end{align*}
\end{proof}

\section{Proof of Theorem \ref{theorem2}}\label{firstproofsection}

Let us start by restating Theorem \ref{theorem2}:

\theoremtwo*

The idea for proving this is essentially to consider two simple kinds of abelian cycles in $H_n(\mathcal{I}_{g,1})$, whose respective images have components in $V_{1^{n+2}}$ and $V_{1^{n}}$ and then taking \textit{products} of such cycles, for a range of $n$'s corresponding to the partitions. The special case of Theorem \ref{theorem1} is given by taking the $n$th power of a bounding pair in $H_1(\mathcal{I}_{3}^1)$ whose image has a component in $V_{1^3}$ (see Example \ref{ExampleSigma1} below).

\subsection{Products of abelian cycles} 

\begin{figure}[htbp]
  \centering
  \includegraphics[scale=0.35]{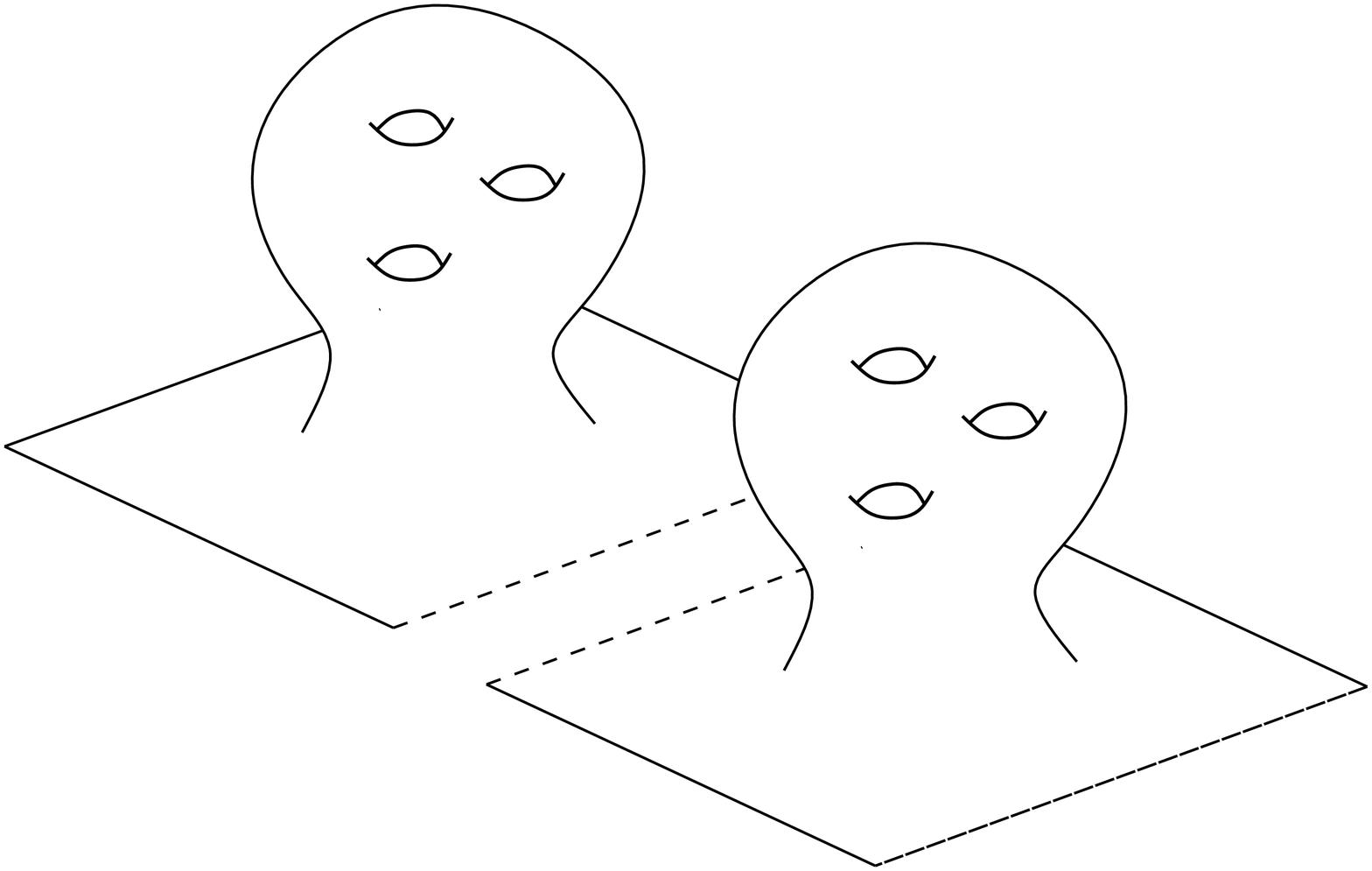}
  \caption{}
  \label{figGlue}
\end{figure}

Given mapping classes $f_1\in\mathcal{I}_{g_1}^1$ and $f_2\in\mathcal{I}_{g_2}^1$, we may glue the surfaces $S_{g_1}^1$ and $S_{g_2}^1$ along the dashed segments illustrated in Figure \ref{figGlue}, and since both mapping classes fix the boundaries, this naturally gives us a mapping class in $\mathcal{I}_{g_1+g_2}^1$. Thus we get a map $\mathcal{I}_{g_1}^1\times\mathcal{I}_{g_2}^1\to\mathcal{I}_{g_1+g_2}^1$, which induces a multiplication
$$H_i(\mathcal{I}_{g_1}^1)\otimes H_j(\mathcal{I}_{g_2}^1)\to H_{i+j}(\mathcal{I}_{g_1+g_2}^1),$$
by precomposing with the Künneth map. This makes the space $\bigoplus_{g,n\ge 0} H_n(\mathcal{I}_{g}^1)$ into a bigraded associative and commutative algebra. 

In $H_{i+j}(\mathbb{Z}^{i+j})\cong H_{i}(\mathbb{Z}^i)\otimes H_j(\mathbb{Z}^j)$, the fundamental class is the tensor product of the fundamental classes in $H_{i}(\mathbb{Z}^i)$ and $H_{j}(\mathbb{Z}^j)$.  This proves the following proposition:

\begin{proposition}\label{PropProdAC}
If $A(f_1,\ldots,f_i)$ and $A(f_{i+1},\cdots,f_{i+j})$ are abelian cycles in $H_i(\mathcal{I}_{g_1}^1)$ and $H_j(\mathcal{I}_{g_2}^1)$ respectively, their product is the abelian cycle $A(f_1,\ldots,f_{i+j})\in H_{i+j}(\mathcal{I}_{g_1+g_2}^1)$.
\end{proposition}

\begin{example}\label{ExampleSigma1}
With this in mind, we see that the abelian cycle depicted in Figure \ref{figAC} is simply the product $\sigma_1^n\cdot 1_{g-3n}$, where $\sigma_1$ is the class in $H_1(\mathcal{I}_{3}^1)$ determined by the bounding pair in Figure \ref{figSigma1}, and $1_{g-3n}\in H_0(\mathcal{I}_{g-3n}^1)$ is the generator. 
\end{example}

\begin{figure}[htbp]
  \centering
  \includegraphics[scale=0.37]{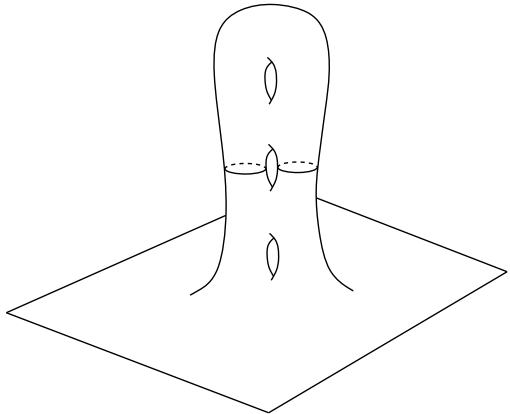}
  \caption{The cycle $\sigma_1$.}\label{figSigma1}
\end{figure}

\subsection{Families of truly nested bounding pairs} With Example \ref{ExampleSigma1} in mind, we will consider abelian cycles determined by families of bounding pairs of the form depicted in Figure \ref{figSigmaN} and then their products. We call these families of bounding pairs \textit{truly nested}, a terminology that we borrow from \cite{Church-Farb-Torelli}.

\begin{definition}\label{defTrulyNested}
    A family $(\alpha_1,\beta_1),\ldots,(\alpha_n,\beta_n)$ of bounding pairs is truly nested if\begin{enumerate}
    \item the curves $\alpha_i$ are pairwise non-homologous and
    \item after possibly reordering the bounding pairs, cutting the surface along $(\alpha_j,\beta_j)$ separates the bounding pair $(\alpha_i,\beta_i)$ from the original boundary component whenever $i<j$.
\end{enumerate}
\end{definition}
We will always assume that the bounding pairs in such a family is numbered so that condition (2) in the definition is met. By induction it can be seen that a truly nested family of bounding pairs is always of the form depicted in Figure \ref{figSigmaN}, but possibly with holes between each adjacent couple of bounding pairs.

\begin{figure}[htbp]
  \centering
  \includegraphics[scale=0.4]{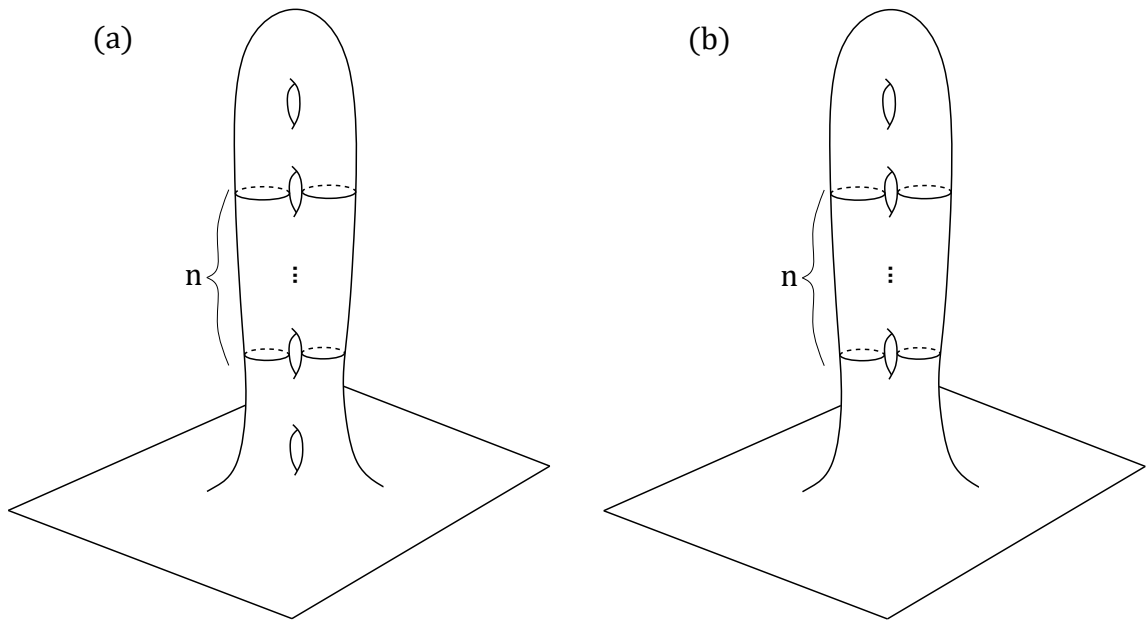}
  \caption{Truly nested families of bounding pairs defining the abelian cycles $\sigma_n\in H_n(\mathcal{I}_{n+2}^1)$ and $\rho_n\in H_n(\mathcal{I}_{n+1}^1)$ considered in Example \ref{exampleCF1}.}
  \label{figSigmaN}
\end{figure}

If we are given a truly nested family of bounding pairs $(\alpha_1,\beta_1),\ldots,(\alpha_n,\beta_n)$ and we cut the surface along a bounding pair, we can talk about the resulting connected component ``furthest'' from the original boundary component. If the surface is cut along $(\alpha_1,\beta_1)$, consider that component and let $\omega_0$ denote the symplectic form of a maximal subsurface with only one boundary component. In order to prove Theorem \ref{theorem2}, we will need the following lemma:

\begin{lemma}\label{lemmaTNF}
Let $A(f_1,\ldots,f_n)$ be an abelian cycle determined by a truly nested family of bounding pair maps. Then there exists an $\mathrm{Sp}(H)$-equivariant map $\Phi_n:\bigwedge^n\left(\bigwedge^3 H\right)\to\bigwedge^{n+2} H$, defined by a sequence of $n-1$ contractions, such that
\begin{equation}
    (\Phi_n\circ\psi_n)(A(f_1,\ldots,f_n))=\lambda_n\omega_0\wedge c_1\wedge\cdots\wedge c_n,\label{eqCFformula}
\end{equation}
for some non-zero integer $\lambda_n$, where $c_i$ is the homology class of the bounding pair corresponding to $f_i$ and $\omega_0$ is defined as above.
\end{lemma}

We construct the map $\Phi_n$ as follows: For $n=1$ we set $\Phi_1=\mathrm{id}_{\bigwedge^3 H}$. For $n\ge 2$, we have a map $\phi_1^n:\bigwedge^n\left(\bigwedge^3 H\right)\to\bigwedge^4 H\otimes\bigwedge^{n-2}\left(\bigwedge^3 H\right)$, given by contracting ``diagonally'' with the symplecic form and composing with the multiplication map $\bigwedge^2\left(\bigwedge^2 H\right)\to\bigwedge^4 H$. For any $2\le k\le n-1$, we similarly have a composite map
$$\bigwedge^{k+2} H\otimes \bigwedge^{n-k}\left(\bigwedge^3 H\right)\to\bigwedge^{k+1} H\otimes \bigwedge^2 H\otimes\bigwedge^{n-k-1}\left(\bigwedge^3 H\right)\to\bigwedge^{k+3} H\otimes\bigwedge^{n-k-1}\left(\bigwedge^3 H\right),$$
which we denote by $\phi_k^n$. We define $\Phi_n:\bigwedge^{n}\left(\bigwedge^3 H\right)\to\bigwedge^{n+2} H$ as the composition $\phi_{n-1}^n\circ\cdots\circ \phi_1^n$.

The calculations needed to prove the formula in Lemma \ref{lemmaTNF} are simple but tedious and not very enlightening, and we will therefore leave them until the end of this section. Before proving the theorem, let us consider the following example, which illustrates the idea of the proof:

\begin{example}\label{exampleCF1}
Let $\sigma_n\in H_n(\mathcal{I}_{n+2}^1)$ denote the abelian cycle determined by the bounding pairs in Figure \ref{figSigmaN}(a). With an appropriate choice of symplectic basis, it follows from Lemma \ref{lemmaTNF} that
$$(\Phi_n\circ\psi_n)(\sigma_n)=\lambda_n a_1\wedge b_1\wedge a_3\wedge\cdots\wedge a_{n+2},$$
for some nonzero integer $\lambda_n$. As in the proof of Theorem \ref{theorem1}, we can now act by $T_{1,2}-\mathbf{1}$, which gives us
$$\lambda_na_1\wedge a_2\wedge\cdots\wedge a_{n+2}.$$
From Corollary \ref{Corollary:Main}, we can now draw the conclusion that $V_{1^{n+2}}$ lies in the image, since this is the top weight representation of $\bigwedge^{n+2} H$. 

Similarly, let $\rho_n\in H_n(\mathcal{I}_{n+1}^1)$ denote the cycle determined by bounding pairs as in Figure \ref{figSigmaN}(b). Let $\Psi_n$ denote the composition of $\Phi_n$ with the contraction map $\bigwedge^{n+2} H\to\bigwedge^n H$. If we once again order the basis appropriately, we see that
$$(\Psi_n\circ\psi_n)(\rho_n)=\lambda_n a_1\wedge a_2\wedge\cdots \wedge a_n,$$
and hence $V_{1^n}$ lies in the image. We can thus draw the conclusion that if $g\ge n+2$, then $V_{1^{n+2}}\oplus V_{1^n}$ are contained in the image of $\psi_n$, while if $g\ge n$, $V_{1^n}$ is contained in the image. 

With $\lambda$ and $\mu$ as in Theorem \ref{theorem2}, we see that for $i=1,\ldots,m$, $\sigma_{\lambda_i}$ and $\rho_{\mu_i}$ each gives us a copy of $V_{1^{\lambda_i+2}}=V_{1^{\mu_i}}$ in $\psi_{\lambda_i}(H_{\lambda_i}(\mathcal{I}_{\lambda_i+2}^1))$ and $\psi_{\mu_i}(H_{\mu_i}(\mathcal{I}_{\mu_i+1}^1))$, respectively. The weight of these irreducibles is at least three. Meanwhile, $\rho_{\mu_{m+1}}$ and $\rho_{\mu_{m+2}}$ give us copies of $V_{1^{\mu_{m+1}}}$ and $V_{1^{\mu_{m+2}}}$ in $\psi_{\mu_{m+1}}(H_{\mu_{m+1}}(\mathcal{I}_{\mu_{m+1}+1}^1))$ and $\psi_{\mu_{m+2}}(H_{\mu_{m+2}}(\mathcal{I}_{\mu_{m+2}+1}^1))$ respectively, and here the weight may be lower.

The idea for Theorem \ref{theorem2} is therefore to consider the product $\sigma_{\lambda_1}^{k_1}\rho_{\mu_1}^{l_1}\cdots\sigma_{\lambda_{m}}^{k_{m}}\rho_{\mu_{m}}^{l_{m}}\rho_{\mu_{m+1}}^{l_{m+1}}\rho_{\mu_{m+2}}^{l_{m+1}}\cdot\mathbf{1}_{g-n-2k-l}$. This is what gives us the condition that 
$$g\ge \sum_{i=1}^{m+2}( k_i(\lambda_i+2)+l_i(\mu_i+1))=|\lambda|+|\mu|+2k+l=n+2k+l.$$

\end{example}

\subsection{Proof of the theorem} Using this we may now prove the theorem:

\begin{proof}[Proof of Theorem \ref{theorem2}, assuming Lemma \ref{lemmaTNF}] As stated previously we will use the abelian cycle
$$\sigma_{\lambda_1}^{k_1}\rho_{\mu_1}^{l_1}\cdots\sigma_{\lambda_{m}}^{k_{m}}\rho_{\mu_{m}}^{l_{m}}\rho_{\mu_{m+1}}^{l_{m+1}}\rho_{\mu_{m+2}}^{l_{m+1}}\cdot\mathbf{1}_{g-n-2k-l},$$
where $\sigma_{\lambda_i}$ and $\rho_{\mu_i}$ are defined as in Example \ref{exampleCF1}. For $1\le i\le m+2$ and $1\le p\le k_i$, we let $H_{\sigma_{\lambda_i}}^{(p)}$ denote the first homology of the subsurface corresponding to the $p$th $\sigma_{\lambda_i}$-factor, with $H_{\sigma_{m+1}}^{(p)}$ and $H_{\sigma_{m+2}}^{(p)}$ defined as 0 for all $p$. We define $H_{\rho_{\mu_i}}^{(q)}$ similarly, so that
$$H=\bigoplus_{i=1}^{m+2}\left(\left(\bigoplus_{p=1}^{k_i} H_{\sigma_{\lambda_i}}^{(p)}\right)\oplus\left(\bigoplus_{q=1}^{l_i} H_{\rho_{\mu_i}}^{(q)}\right)\right)\oplus H_1(S_{g-n-2k-l}).$$
By Proposition \ref{tauformula} we have 
$$\psi_{n}(\sigma)\in \bigotimes_{i=1}^{m+2}\left(\left(\bigotimes_{p=1}^{k_i}\bigwedge^{\lambda_i}\bigwedge^3 H_{\sigma_{\lambda_i}}^{(p)}\right)\otimes\left(\bigotimes_{q=1}^{l_i}\bigwedge^{\mu_i}\bigwedge^3 H_{\rho_{\mu_i}}^{(q)}\right)\right)\subseteq \bigwedge^{n}\bigwedge^3 H$$
Let $F:=\Phi_{\lambda_1}^{\wedge k_1}\wedge\Psi_{\mu_1}^{\wedge l_1}\wedge\cdots\wedge\Phi_{\lambda_{m}}^{\wedge k_{m}}\wedge\Psi_{\mu_{m}}^{\wedge l_{m}}\wedge\Psi_{\mu_{m+1}}^{\wedge l_{m+1}}\wedge \Psi_{\mu_{m+2}}^{\wedge l_{m+2}}$. Since the direct summands $H_{\sigma_{\lambda_i}}^{(p)}$ and $H_{\rho_{\mu_i}}^{(q)}$ of $H$ are pairwise orthogonal with respect to the symplectic form $\omega$, it follows from the definitions of $\Phi_{\lambda_i}$ and $\Psi_{\lambda_i}$ that
$$F(\psi_n(\sigma))\in\bigotimes_{i=1}^{m+2}\left(\left(\bigotimes_{p=1}^{k_i}\bigwedge^{\mu_i} H_{\sigma_{\lambda_i}}^{(p)}\right)\otimes\left(\bigotimes_{q=1}^{l_i}\bigwedge^{\mu_i}H_{\rho_{\mu_i}}^{(q)}\right)\right)\subseteq\bigotimes_{i=1}^{m+2}\bigwedge^{k_i+k_i}\left(\bigwedge^{\mu_{i}} H\right) $$
Since the top weight irreducibles of $\bigotimes_{i=1}^{m+2}\bigwedge^{k_i+l_i}\left(\bigwedge^{\mu_{i}} H\right)$ are precisely the top weight irreducibles of $\bigotimes_{i=1}^{m+2}\bigwedge^{k_i+l_i} V_{1^{\mu_i}}$, we now want to apply Corollary \ref{Corollary:Main}. The representation $\bigotimes_{i=1}^{m+2}\bigwedge^{k_i+l_i}\left(\bigwedge^{\mu_{i}} H\right)$ is a direct summand of $H^{\otimes (n+2k)}$. Let us denote the projection map by $\pi$. If we choose a symplectic basis for each $H_{\sigma_{\lambda_i}}^{(p)}$ and each $H_{\sigma_{\mu_i}}^{(q)}$, extend the union of these to a symplectic basis $\{a_1,b_1,\ldots,a_g,b_g\}$ for $H$ and order this basis appropriately, we have
$$\pi(a_1\otimes\cdots\otimes a_{n+2k})\in\bigotimes_{i=1}^{m+2}\left(\left(\bigotimes_{p=1}^{k_i}\bigwedge^{\mu_i} H_{\sigma_{\lambda_i}}^{(p)}\right)\otimes\left(\bigotimes_{q=1}^{l_i}\bigwedge^{\mu_i}H_{\rho_{\mu_i}}^{(q)}\right)\right)\subseteq\bigotimes_{i=1}^{m+2}\bigwedge^{k_i+l_i}\left(\bigwedge^{\mu_{i}} H\right).$$
From what we saw in Example \ref{exampleCF1}, there is an element 
$$T\in\mathbb{Q}\left[\left(\prod_{i=1}^{m+2}\left(\prod_{p=1}^{k_i}\mathrm{Sp}(H_{\sigma_{\lambda_i}}^{(p)})\right)\times\left(\prod_{q=1}^{l_i}\mathrm{Sp}(H_{\rho_{\mu_i}}^{(q)})\right)\right)\times \mathrm{Sp}(H_1(S_{g-n-2k-l}))\right]\subseteq \mathbb{Q}[\mathrm{Sp}(H)]$$
such that
\begin{align*}
    T\cdot(F(\psi_n(\sigma)))=\pi (a_1\otimes\cdots\otimes a_{n+2k})&\qedhere
\end{align*}
\end{proof}

\subsection{Proof of Lemma \ref{lemmaTNF}} Now let us prove the lemma.

\begin{proof}[Proof of Lemma \ref{lemmaTNF}]

We prove the formula (\ref{eqCFformula}) by induction on $n$. For $n=1$ we have $\psi_1=\tau$ and $\Phi_1=\mathrm{Id}_{\bigwedge^3 H}$, so the identity holds with $\lambda_1=1$, by Proposition \ref{tauformula}.

Now suppose $n\ge 2$ and that $f_1,\ldots,f_n$ are bounding pair maps that correspond to truly nested bounding pairs $(\alpha_1,\beta_1),\ldots,(\alpha_n,\beta_n)$. We may pick a symplectic basis $\{a_1,b_1,\ldots,a_g,b_g\}$ such that $[\alpha_k]=a_{i_k}$ for $1\le k\le n$, where $i_1\le \cdots\le i_n$ and such that $a_1\wedge b_1+\cdots+a_{i_k-1}\wedge b_{i_k-1}$ descends to a symplectic form of the homology of the surface not containing the original boundary component that we get by cutting along $(\alpha_k,\beta_k)$. For $n=2$, a simple calculation shows that
\begin{align*}
    \Phi_2\circ\psi_2(A(f_1,f_2))&=\phi_{1}^2\Big(\big((a_1\wedge b_1+\cdots+a_{i_1-1}\wedge b_{i_1-1})\wedge a_{i_1}\big)\wedge\big((a_1\wedge b_1+\cdots+a_{i_2-1}\wedge b_{i_2-1})\wedge a_{i_2}\big)\Big)\\
    &=-3(a_1\wedge b_1+\cdots+a_{i_1-1}\wedge b_{i_1-1})\wedge a_{i_1}\wedge a_{i_2}.
\end{align*}
Now we let $n\ge 3$ and assume that 
$$\Phi_{n-1}\circ\psi_{n-1}(A(f_1,\ldots,\hat{f_i},\ldots ,f_{n}))=\lambda_{n-1}(a_1\wedge b_1+\cdots+a_{i_1-1}\wedge b_{i_1-1})\wedge a_{i_1}\wedge\cdots\wedge\hat{a}_{i_k}\wedge \cdots\wedge a_{i_n}$$
and
$$\Phi_{n-1}\circ\psi_{n-1}(A(f_2,\ldots,f_{n}))=\lambda_{n-1}(a_1\wedge b_1+\cdots+a_{i_2-1}\wedge b_{i_2-1})\wedge a_{i_2}\wedge \cdots\wedge a_{i_n},$$
for some non-zero $\lambda_n\in\mathbb{Z}$. Since
$$\Phi_n=\phi_{n-1}^n\circ(\Phi_{n-1}\wedge\mathrm{Id}_{\bigwedge^3 H})$$
and 
$$\psi_n(A(f_1,\ldots,f_n))=(-1)^{n-k}\psi_{n-1}(A(f_1,\ldots,\hat{f}_k,\ldots,f_n)\wedge \tau([f_k])$$
for each $k=1,\ldots,n$, it follows that
\begin{gather*}
    \Phi_{n}\circ\psi_n(A(f_1,\ldots,f_n))=\\
    \lambda_{n-1}\cdot \phi_{n-1}^n\bigg((-1)^{n-1}(a_1\wedge b_1+\cdots+a_{i_2-1}\wedge b_{i_2-1})\wedge a_{i_2}\wedge \cdots\wedge a_{i_n}\otimes (a_1\wedge b_1+\cdots+a_{i_1-1}\wedge b_{i_1-1})\wedge a_{i_1}\\
    +\sum_{k=2}^n (-1)^{n-k}(a_1\wedge b_1+\cdots+a_{i_1-1}\wedge b_{i_1-1})\wedge a_{i_1}\wedge\cdots\wedge\hat{a}_{i_k}\wedge \cdots\wedge a_{i_n}\otimes (a_1\wedge b_1+\cdots+a_{i_k-1}\wedge b_{i_k-1})\wedge a_{i_k}\bigg).
\end{gather*}
A simple calculation shows that
$$\phi_{n-1}^n\bigg((a_1\wedge b_1+a_{i_2-1}\wedge b_{i_2-1})\wedge a_{i_2}\wedge \cdots\wedge a_{i_n}\otimes(a_1\wedge b_1+\cdots+a_{i_1-1}\wedge b_{i_1-1})\wedge a_{i_1}\bigg)$$
$$=3(-1)^{n}(a_1\wedge b_1+\cdots+a_{i_1-1}\wedge b_{i_1-1})\wedge a_{i_1}\wedge\cdots\wedge a_{i_n}$$
and similarly
$$\phi_{n-1}^n\bigg((a_1\wedge b_1+\cdots+a_{i_1-1}\wedge b_{i_1-1})\wedge a_{i_1}\wedge\cdots\wedge\hat{a}_{i_k}\wedge \cdots\wedge a_{i_n}\otimes (a_1\wedge b_1+\cdots+a_{i_k-1}\wedge b_{i_k-1})\wedge a_{i_k}\bigg)$$
$$=(-1)^{n-k-1}(k+1)(a_1\wedge b_1+\cdots+a_{i_1-1}\wedge b_{i_1-1})\wedge a_{i_1}\wedge\cdots\wedge a_{i_n}.$$
Thus we have
\begin{align*}
    \Phi_n\circ\psi_n(A(f_1,\ldots,f_n))&=    -\lambda_{n-1}\sum_{k=0}^{n}(k+1)(a_1\wedge b_1+\cdots+a_{i_1-1}\wedge b_{i_1-1})\wedge a_{i_1}\wedge\cdots\wedge a_{i_n}\\
    &=-\lambda_{n-1}\frac{(n+1)(n+2)}{2}\tau_n(A(f_1,\ldots,f_n)),
\end{align*}
which proves the formula. \end{proof}

\begin{remark}
The idea to consider truly nested families of bounding pairs is taken from \cite{Church-Farb-Torelli}. There the authors consider a family of $\mathrm{Sp}(H_{\mathbb{Z}})$-equivariant maps $\tau_n:H_n(\mathcal{I}_{g,1})\to\bigwedge^{n+2} H$ that generalize the map $\tau:H_1(\mathcal{I}_{g,1})\to\bigwedge^3 H$ to higher degree. We can note that $\tau_n$ and $\Phi_n\circ\psi_n$ have the same codomain. In \cite[Section 3]{Church-Farb-Torelli}, the image of $\tau_n$ on abelian cycles determined by truly nested families of bounding pair maps is proven to satisfy
$$\tau_n(A(f_1,\ldots,f_n))=\omega_0\wedge c_1\wedge\cdots\wedge c_n,$$
with the notation being the same as in our Lemma \ref{lemmaTNF}. We can thus see that on such abelian cycles, the two maps agree, up to multiplication by a non-zero scalar. An interesting question is if these maps agree on all of $H_n(\mathcal{I}_{g,1})$, so that the image of $\tau_n$ is simply a subrepresentation of the image of $\psi_n$, or if there is some other way to explicitly relate $\psi_n$ and $\tau_n$.
\end{remark}

\section{Proof of Theorem \ref{theorem3} }\label{secondproofsection}

As we have seen, abelian cycles determined by disjoint bounding pair maps are a useful and simple tool for computing big parts of the image of $\psi_n$. In this section we will give an explicit limitation on the images of such cycles. More specifically, we will prove the following theorem:

\theoremthree*

We will prove the following two claims of the theorem separately:

\begin{claim}\label{claim1}
    The image $\psi_n(A_n(\mathcal{I}_{g,1}))$ is concentrated in weights $n,n+2,\ldots,3n$.
\end{claim}
\begin{claim}\label{claim2}
For $n\ge 2$ and $g\gg 0$, the image of $\psi_n(H_n(\mathcal{I}_{g,1}))$ contains an irreducible subrepresentation of weight $n-2$.
\end{claim}

To prove Claim \ref{claim1}, we will need to find a description of the tensor $\psi_n(A(f_1,\ldots,f_n))$ that we can work with, for any $A(f_1,\ldots,f_n)\in A_n(\mathcal{I}_{g,1})$. This will be done using the following lemma:

\begin{lemma}\label{lemmaLagrangian}
    Let $A(f_1,\ldots, f_n)\in A_n(\mathcal{I}_{g,1})$. Then there exists a Lagrangian subspace $L\subset H$, such that for $k=1,\ldots,n$, we have
    $$\tau([f_k])\in\mathrm{Im}\left(\bigwedge^2 L\otimes H\to \bigwedge^3 H\right).$$
\end{lemma}

\begin{proof}

Let us denote the bounding pair corresponding to $f_k$ by $(\alpha_k,\beta_k)$, for $k=1,\ldots, n$. Note that since the bounding pairs are pairwise disjoint, the span of $\{[\alpha_1],\ldots,[\alpha_n]\}$ is an isotropic subspace of $H$, which we may denote by $I$. Recalling Proposition \ref{tauformula}, we want to extend this to a Lagrangian subspace $L$ in such a way that after the surface is cut by $(\alpha_k,\beta_k)$, for any $k$, $L$ restricts to a Langrangian subspace of the first homologies of both of the closed surfaces obtained by filling in the boundary components with open disks.

Now let us cut the surface along all the bounding pairs and sew in open disks in all resulting boundary components, in order to get a $n+1$ closed surfaces. For each such closed surface we may pick curves whose homology classes span a Lagrangian subspace of its first homology. These curves can be chosen so that they do not intersect the sewn-in disks or their boundaries and thus, viewed as curves in $S_g$ they do not intersect any of our bounding pairs. This means that taking the span of the homology classes of these curves and all bounding pairs will still be an isotropic subspace of $H$. Let us denote this subspace by $L$. To see that $L$ is Lagrangian, it suffices to note that the total genus of the subsurfaces we get by cutting along all bounding pairs is $g-\dim I$, which means that this is also the dimension of the subspace of $H$ spanned by the chosen curves. Since their classes are linearly independent of the classes in $I$, it follows that $\dim L=g$.

By the construction of $L$ through a selection of pairwise non-intersecting curves, it is clear that if the surface is cut along one bounding pair, the remaining curves all end up in only one of the connected components and are still pairwise non-intersecting. Since the genera of the resulting two subsurfaces add to $g-1$, it hence follows that the span of the homology classes of the curves in either boundary component still form Lagrangian subspaces in their respective components, and hence $L$ has the sought after property. 

This means that when the surface is cut along the bounding pair $(\alpha_k,\beta_k)$, for any $k=1,\ldots,n$, we can find a symplectic basis for the first homology of the closed surface corresponding to the connected component not containing the original boundary component, such that its symplectic forms can be written 
$$\omega_1=\sum_{i=1}^{g_1} a_i\wedge b_i,$$
with $a_i\in L$. It follows by Proposition \ref{tauformula}, that
$$\tau([f_k])=\omega_1\wedge [\alpha_k]$$
which lies in $\mathrm{Im}\left(\bigwedge^2 L\otimes H\to \bigwedge^3 H\right)$, since $[\alpha_k]\in L$. \end{proof}

Using this we may now prove the claim:
\begin{proof}[Proofof Claim \ref{claim1}]

By Lemma \ref{lemmaLagrangian}, we can find a Lagrangian $L\in H$ such that 
$$\psi_n(A(f_1,\ldots,f_n))\in\mathrm{Im}\left(\bigwedge^n\left(\bigwedge^2 L\otimes H\right)\to\bigwedge^n\left(\bigwedge^3 H\right)\right),$$
This means that if we apply the standard inclusion $\iota:\bigwedge^n\left(\bigwedge^3 H\right)\hookrightarrow H^{\otimes 3n}$, each term of $\psi_n(A(f_1,\ldots,f_n))$ lies in the subspace
$$L^{\otimes{2n}}\otimes H^{\otimes n}.$$
up to a permutation of the factors. From this it follows that $\iota(\psi_n(A(f_1,\ldots,f_n)))$ must be mapped to zero by any composition of more than $n$ contractions and thus it cannot have a component in an irreducible subrepresentation of weight less than $n$. \end{proof}

\begin{remark}\label{remSakasai}
We have stated Claim \ref{claim1} in the pointed case, but it holds in the unpointed case as well. To see this, let $U:=\left(\bigwedge^3 H\right)/H$, $A_n(\mathcal{I}_{g})\subseteq H_n(\mathcal{I}_g)$ be defined similarly as $A_n(\mathcal{I}_{g,1})$ and let $\psi_n':H_n(\mathcal{I}_g)\to\bigwedge^n U$ be defined similarly as $\psi_n$. We then have the commutative diagram
\[\begin{tikzcd}
A_n(\mathcal{I}_{g,1})\arrow[r,hook]\arrow[d,two heads]&H_n(\mathcal{I}_{g,1})\arrow[r,"\psi_n"]\arrow[d]&\bigwedge^n\bigwedge^3 H\arrow[d,two heads]\\
A_n(\mathcal{I}_{g})\arrow[r,hook]&H_n(\mathcal{I}_{g})\arrow[r,"\psi_n'"]&\bigwedge^n U
\end{tikzcd}\]
which proves that $\psi_n'(A_n(\mathcal{I}_{g}))$ cannot contain any irreducibles of weight less than $n$. In \cite{Sakasai-Torelli}, Sakasai uses\footnote{He also uses one abelian cycle of a different kind, but whose image under $\psi_3'$ can easily be generated using elements of $A_3(\mathcal{I}_{g})$ as it is concentrated in top weight} elements of $A_3(\mathcal{I}_{g})$ to determine the stable image of (the dual map of) $\psi_3'$, but his result leaves open whether or not the image contains a copy of $V_1$. Claim \ref{claim1} illustrates why this could not be determined using Sakasai's method.
\end{remark}

Now it only remains to prove Claim \ref{claim2}. Before proving this we need some additional background. Let $\mathcal{T}_g$ denote \textit{Torelli space}, the moduli space of Riemann surfaces of genus $g$ endowed with a homology marking, i.e. a choice of basis for the first homology group. This is the quotient of the Teichmüller space $\mathrm{Teich}_g$ by the free action of $\mathcal{I}_g$, which means that $H_*(\mathcal{I}_g)\cong H_*(\mathcal{T}_g)$. In a similar way we can construct the Torelli space $\mathcal{T}_{g,1}$ of pointed Riemann surfaces. The map given by forgetting the marked point, which we will denote by $\pi$, makes
$$S_g\to\mathcal{T}_{g,1}\overset{\pi}{\to}\mathcal{T}_{g}$$
into a universal $S_g$-bundle over $\mathcal{T}_g$. We will use this to prove the lemma.

\begin{proof}[Proof of Claim \ref{claim2}.]
First, we consider the case $n=2$. Let $[*]\in H_0(\mathcal{T}_g)$ be the class given by choosing a point $*\hookrightarrow \mathcal{T}_g$. Pulling this map back along the universal bundle, we get the trivial bundle
\[\begin{tikzcd}
S_g\arrow[d]\arrow[r]&\mathcal{T}_{g,1}\arrow[d,"\pi"]\\
*\arrow[r]&\mathcal{T}_g
\end{tikzcd}\]
which by integration over the fibers gives us a class $\pi^!([*])$ in $H_2(S_g)$, which is just the fundamental class $[S_g]$. By abuse of notation, we will use $[S_g]$ to denote the corresponding class $\pi^!([*])\in H_2(\mathcal{T}_{g,1})\cong H_2(\mathcal{I}_{g,1})$ as well. 

\setlength{\abovedisplayskip}{5pt}
\setlength{\belowdisplayskip}{5pt}

To compute $\psi_2([S_g])$, we may compute its coproduct in $H_2(S_g)$ and then use functoriality. The coproduct can be computed using the diagram
\[\begin{tikzcd}
H_2(S_g)\arrow[r]\arrow[d]& H_2(Z)\arrow[d]\\
\bigwedge^2 H_1(S_g)\arrow[r]&\bigwedge^2 H_1(Z),
\end{tikzcd}\]
where $Z$ is the wedge of tori in Figure \ref{figQuotientSurface}, that we get from $S_g$ by collapsing a subsurface of genus $0$ with $g$ boundary components to a point. 
\begin{figure}[htbp]
  \centering
  \includegraphics[scale=0.4]{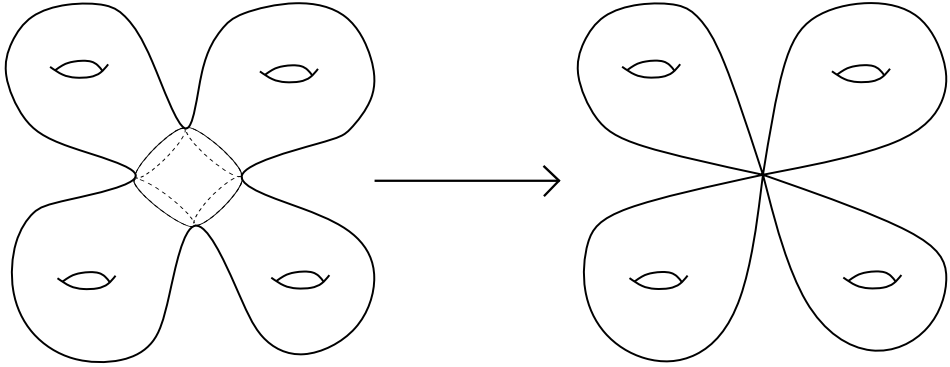}
  \caption{The quotient surface obtained from $S_g$ by collapsing a subsurface of genus $0$ with $g$ boundary components to a point.}
  \label{figQuotientSurface}
\end{figure}
The horizontal arrow on the second row of the diagram is an isomorphism, so we only need to compute the coproduct of the fundamental class $[Z]$, which is equal to the sum $[Z_1]+\cdots+[Z_g]\in H_2(Z)\cong H_2(Z_1)\oplus\cdots\oplus H_2(Z_g)$ of the fundamental classes of the $g$ tori. Since the homology of each torus is a bialgebra, we get $\Delta([Z_k])=a_k\wedge b_k$, where $\{a_k,b_k\}$ is a symplectic basis for $H_1(Z_k)$. The union of these bases correspond in a natural way to a symplectic basis $\{a_1,b_1,\ldots,a_g,b_g\}$ of $H_1(S_g)$, so we get that $\Delta([S_g])=\omega$, where $\omega$ is the symplectic form of $H_1(S_g)$.

The map $S_g\to \mathcal{T}_{g,1}$ induces the so called ``point-pushing'' map $\pi_1(S_g)\to\mathcal{I}_{g,1}$ \cite[Chapters 4.2, 6.4.2]{FarbMargalitPrimer} and using this it follows that if $c\in H_1(S_g)$ and we denote its image by $[c]\in H_1(\mathcal{I}_{g,1})$, then $\tau([c])=\omega\wedge c\in \bigwedge^3 H$. Thus it follows by functoriality that $\psi_2([S_g])$ is the image of $\omega$ by the map $\bigwedge^2 H\to \bigwedge^2\left(\bigwedge^3 H\right)$, given by inserting the symplectic form in each factor. The class $\omega\in \bigwedge^2 H$ generates the trivial representation, which means that $\psi_2([S_g])$ has weight 0. This proves the lemma in the case $n=2$.

The case $n=3$ has already been treated by Kupers and Randal-Williams \cite[Section 8]{Kupers-Torelli}, who have proven that for $g\gg 0$ the image of the dual of the map $\psi_3:H_3(\mathcal{I}_{g}^1)\to\bigwedge^3\left(\bigwedge^3 H\right)$ contains $V_1$, which implies that $V_1$ is also contained in the image of $\psi_3:H_3(\mathcal{I}_{g,1})$, by the factorization $\mathcal{I}_{g}^1\to\mathcal{I}_{g,1}\to\bigwedge^3 H_\mathbb{Z}$.

To prove the result for $n\ge 4$, we will use the result for $n=2$. Let us suppose that $n\ge 4$ and $g\ge 2n-4$. If we once again let $U:=\left(\bigwedge^3 H\right)/H$, we have a commutative diagram
\[\begin{tikzcd}
H_n(\mathcal{T}_{g,1})\arrow[r]&\bigwedge^{n-2} H_1(\mathcal{T}_{g,1})\otimes H_2(\mathcal{T}_{g,1})\arrow[r]\arrow[d,"\pi_*\otimes\mathrm{id}"]&\bigwedge^n H_1(\mathcal{T}_{g,1})\arrow[r,"\cong"]\arrow[d]&\bigwedge^n\bigwedge^3 H\arrow[d,two heads]\\
&\bigwedge^{n-2} H_1(\mathcal{T}_g)\otimes H_2(\mathcal{T}_{g,1})\arrow[r]&\bigwedge^{n-2} H_1(\mathcal{T}_g)\otimes\bigwedge^2 H_1(\mathcal{T}_{g,1})\arrow[r,"\cong"]&\bigwedge^{n-2} U\otimes \bigwedge^2\bigwedge^3 H\\
H_{n-2}(\mathcal{T}_g)\arrow[uu,"\pi^!"]\arrow[r]&\bigwedge^{n-2}H_1(\mathcal{T}_{g})\otimes H_0(\mathcal{T}_g)\arrow[u,"\mathrm{id}\otimes \pi^!"']
\end{tikzcd}\]
where the leftmost square commutes by the projection formula and the middle square commutes by functoriality of homology.

The composition of the arrows in the upper row is $\psi_n$. Since the rightmost vertical arrow is surjective, and since we now know that the image of the composition $\psi_2\circ\pi^!:H_0(\mathcal{T}_g)\to\bigwedge^2\bigwedge^3 H$ lies in weight zero, it now suffices to find a class in $H_{n-2}(\mathcal{T}_g)$ whose image in $\bigwedge^{n-2} U$ has a non-zero component in weight $n-2$. Note that this is impossible for $n=3$, since $H_1(\mathcal{T}_g)\cong U\cong V_{1^3}$, which is why this case had to be treated separately.

For $m=n-2\ge 2$, let us consider the image of the cycle $\rho_1^m$, with $\rho_1$ defined as in Example \ref{exampleCF1}, under the map $H_m(\mathcal{I}_{g}^1)\to H_m(\mathcal{I}_{g,1})$. We also get a corresponding abelian cycle in $H_m(\mathcal{I}_g)$ by forgetting the marked point. Abusing notation, we will use $\rho_1^m$ to denote this class in either of these homology groups. Let $\psi_m':H_m(\mathcal{I}_g)\to\bigwedge^3 U$ be defined as in Remark \ref{remSakasai}. Then we have a commutative diagram
\[\begin{tikzcd}
H_m(\mathcal{I}_{g,1})\arrow[r,"\psi_m"]\arrow[d]&\bigwedge^m\bigwedge^3 H\arrow[d]\\
H_m(\mathcal{I}_g)\arrow[r,"\psi_m'"]&\bigwedge^m U
\end{tikzcd}\]
which means that we may work with the image of $\psi_m(\rho_1^m)$ in $\bigwedge^m U$. The map $\bigwedge^3 H\to U$ can be described explicitly, by identifying $U$ with the image of the map $p:\bigwedge^3 H\to\bigwedge^3 H$ defined by
$$p(x\wedge y\wedge z)=x\wedge y\wedge z -\frac{1}{g-1}C_3(x\wedge y\wedge z)\wedge\omega,$$
where $C_3:\bigwedge^3 H\to H$ is the contraction map. Note that this image is precisely the subrepresentation of traceless tensors in $\bigwedge^3 H$. With a suitable choice of basis of $H$, we get 
$$(p^{\wedge m}\circ\psi_m)(\rho_1^m)=(a_1\wedge b_1\wedge a_{m+1}-\frac{1}{g-1}\omega\wedge a_{m+1})\wedge\cdots\wedge (a_{m}\wedge b_m\wedge a_{2m}-\frac{1}{g-1}\omega\wedge a_{2m})$$
We want to show that this has a nonzero component in an irreducible subrepresentation of weight $m$, so we apply the map $\bigwedge^m\left(\bigwedge^3 H\right)\to\bigwedge^{3m} H$ and then contract $m$ times. If we let $[m]=\{1,\ldots,n\}$, the image in $\bigwedge^{3m} H$ can be written
\begin{align*}
   \left(\sum_{k=0}^m\frac{(-1)^{k}}{(g-1)^k}\sum_{\{i_1,\ldots,i_k\}\subset [m]}\omega^{\wedge k}\wedge a_1\wedge b_1\wedge \cdots\widehat{a_{i_1}\wedge b_{i_1}}\wedge\cdots\wedge\widehat{a_{i_k}\wedge b_{i_k}}\wedge\cdots a_m\wedge b_m\right)\wedge a_{m+1}\wedge\cdots\wedge a_{2m}
\end{align*}
By contracting each term $m$ times, we get
\begin{align*}
    &\frac{1}{(g-1)^m}\sum_{k=0}^m(-1)^k\binom{m}{k}(g-1)^{m-k}(g-2m+1)^k a_{m+1}\wedge\cdots\wedge a_{2m}=\frac{(2m-2)^m}{(g-1)^m}a_{m+1}\wedge\cdots\wedge a_{2m},
\end{align*}
which is nonzero since $m=n-2\ge 2$ and which contracts to zero and thus lies in $V_{1^m}\subset\bigwedge^m H$. This proves the lemma for $n\ge 4$.\end{proof}

\begin{remark}
The proof of Claim \ref{claim2} is heavily inspired by \cite[Section 4]{Church-Farb-Torelli}, where the authors compute the image of $\tau_n\circ\pi^!$ for truly nested families of bounding pair maps. In our situation, it would be reasonable to expect the formula
$$\psi_{n}(\pi^!\pi_*A(f_1,\ldots,f_{n-2}))=\tau([f_1])\wedge\cdots\wedge\tau([f_{n-2}])\wedge\psi_2([S_g])$$
to hold for any abelian cycle $A(f_1,\ldots,f_{n-2})\in H_{n-2}(\mathcal{I}_{g,1})$, which would prove Claim \ref{claim2} for $n\ge 3$ directly from the case $n=2$. We have, however, not been able to verify or disprove this.
\end{remark}

\printbibliography

\end{document}